\documentclass[a4paper, 12pt]{amsart}
\usepackage[english]{babel}
\usepackage{amsmath}
\usepackage{amssymb}
\usepackage{amscd}
\usepackage{amsthm}
\usepackage{euscript}
\usepackage[colorlinks=true, allcolors=blue]{hyperref}

\usepackage[letterpaper,top=2cm,bottom=2cm,left=3cm,right=3cm,marginparwidth=1.75cm]{geometry}
\usepackage{graphicx}
\newtheorem{proposition}{Proposition}
\newtheorem{lemma}{Lemma}
\newtheorem{theorem}{Theorem}

\newtheorem{corollary}{Corollary}
\newtheorem{conjecture}{Conjecture}
\theoremstyle{definition}
\newtheorem{definition}{Definition}
\newtheorem{example}{Example}
\theoremstyle{remark}
\newtheorem {remark}{Remark}

\DeclareMathOperator{\Aut}{Aut}

\DeclareMathOperator{\supp}{Supp}

\def\Ker{{\rm Ker}\,}

\def\rk{{\mathrm{rk}}}

\def\ZZ{{\mathbb Z}}

\def\BK{{\mathbb K}}

\def\BZ{{\mathbb Z}}

\def\BN{{\mathbb N}}

\title{On automorphism group of toral varieties}
\author{Anton Shafarevich  and Anton Trushin}

\begin{document}
\maketitle

\begin{abstract}
Let $\BK$ be an uncountable algebraically closed field of characteristic zero. An affine algebraic variety $X$ over $\BK$ is toral if it is isomorphic to a closed subvariety of a torus $(\BK^*)^d$. We study the group $\Aut(X)$ of regular automorphisms of a toral variety $X$. We prove that if $T$ is a maximal torus in $\Aut(X)$, then $X$ is a direct product $Y\times T$, where $Y$ is a toral variety with a trivial maximal torus in the automorphism group. We show that knowing $\Aut (Y)$, one can compute $\Aut(X)$. In the case when the rank of the group $\BK[Y]^*/\BK^*$ is $\dim Y + 1$, the group $\Aut(Y)$ is described explicitly. 
\end{abstract}

\section{Introduction}

Let $\BK$ be an algebraically closed field of characteristic zero. The set of solutions of a system of polynomial equations in affine space has been studied for a very long time. But some interesting properties may appear when we consider the set of solutions inside a torus $(\BK^*)^d$. In other words, we consider only solutions with non-zero coordinates. One of the examples of this approach is the Bernstein-Kushnirenko Theorem; see \cite{Ber, Kush}.

\smallskip Vladimir Popov in \cite{Popov1} proposed the following definition.

\begin{definition}
An irreducible affine algebraic variety $X$ is called \emph{toral} if it is isomorphic to a closed subvariety of a torus $(\BK^*)^d$.
\end{definition}

Some authors also use the term a "very affine variety"; see \cite{Tev, Huh}. It can be seen that $X$ is toral if and only if the algebra of regular functions on $X$ is generated by invertible functions; see \cite[Lemma 1.14]{Popov1}. One of the reasons why toral varieties are interesting is that they are rigid varieties; see \cite[Lemma 1.14]{Popov1}.

\begin{definition}
An affine algebraic variety $X$ is called \emph{rigid} if there is no non-trivial action of the additive group $(\BK, +)$ on $X.$

\end{definition}

Despite the fact that the automorphism group of an affine algebraic variety has a complicated structure, sometimes it is possible to describe it for rigid varieties. It was proven in \cite{AGRV} that the group of regular automorphisms $\Aut(X)$ of a rigid variety $X$ contains a unique maximal torus $T$. One can find examples of computation of $\Aut(X)$ for rigid varieties in \cite{AGRV, Perep1, PZ}.

In this paper, we study the automorphism group $\Aut(X)$ of a toral variety $X$. We denote by $\BK[X]$ the algebra of regular functions on $X$ and by $\BK[X]^*$ the multiplicative group of invertible regular functions on $X$. Let $E(X)$ be the quotient group $\BK[X]^*/\BK^*.$ By \cite{Ro}, the group $E(X)$ is a free finitely generated abelian group. For a toral variety $X$ the rank of $E(X)$ is not less than $\dim X.$

Any automorphism of $X$ induces an automorphism of $E(X).$ So we obtain a homomorphism from $\Aut(X)$ to $\Aut(E(X)).$ We denote by $H(X)$ the kernel of this homomorphism. Note that $H(X)$ consists of automorphisms that multiply invertible functions by constants. 

Suppose that $X$ is a closed subvariety of a torus $T_d = (\BK^*)^d$. In Proposition \ref{Prop1} we show that the group $H(X)$ is naturally isomorphic to a subgroup in $T_d$ which consists of elements that preserve $X$ under the action by multiplication. In Proposition \ref{Prop3} we propose a way to compute the subgroup $H(X)$.

In Theorem \ref{theorem1} we show that if $T$ is a maximal torus in $\Aut(X)$ then $X$ is isomorphic to a direct product $T \times Y$ where $Y$ is a toral variety with a discrete automorphism group. Here and below we assume that the field $\BK$ is uncountable. Theorem \ref{theorem2} gives a way to find $\Aut(X)$ knowing $\Aut(Y).$ If the rank of $E(Y)$ is $\dim Y + 1$ it is possible to describe $\Aut(Y)$ (Theorem \ref{Gaff}). 

Also, we consider the case when the rank of $E(X)$ is equal to $\dim X.$ By Proposition \ref{theorem2} in this case $X$ is a torus. Moreover, it is the only case when $\Aut(X)$ acts on $X$ with an open orbit.  

We use the following notation. If $\varphi$ is a regular automorphism of an affine variety $X$ then by $\varphi^*$ we mean an automorphism of $\BK[X]$ dual to 
$\varphi$. If $A$ is a group and $B$ is a normal subgroup in $A$ then by $[a]$ we denote the image of an element $a\in A$ in the quotient group $A/B.$ If $X$ is a closed subvariety of an affine variety $Z$ then by $I(X)$ we mean the ideal of regular functions on $Z$ which are equal to zero on $X$.

\section{General facts on toral varieties}

Here we prove some initial properties of toral varieties and propose a way to compute the group $H(X)$ for a toral variety $X$

Let $T_r$ be a torus of dimension $r$. We recall that the group $\Aut (T_r)$ is isomorphic to $T_r  \rtimes \mathrm{GL}_r(\ZZ)$; see\cite[Example 2.3]{AGRV}. Here the left factor $T_r$ acts on itself by multiplications and a matrix $(a_{ij}) \in \mathrm{GL}_r(\ZZ)$ defines an automorphism of $T_r$ which is given by the formula

$$t_i \to t_1^{a_{i1}}\ldots t_r^{a_{ir}},$$
where $t_1,\ldots, t_r$ are coordinate functions on $T_r.$

Now let $X$ be a toral variety and $r$ is the rank of $E(X)$. One can choose invertible functions $f_1, \ldots, f_r \in \BK[X]^{*}$ such that $[f_1], \ldots, [f_r]$ form a basis of the group $E(X)$. Then $f_1, \ldots, f_r$ generate the algebra $\BK[X]$ and define a closed embedding of $\rho: X \hookrightarrow T_r.$ Note that if we choose another $g_1,\ldots, g_r \in \BK[X]^{*}$ such that $[g_1],\ldots, [g_r]$ form a basis of $E(X)$ then the respective embedding $\rho_g: X \hookrightarrow T_r$ differs from $\rho$ by an automorphism of $T_r$. Indeed, we have 
$$g_i = \lambda_if_1^{a_{i1}}\ldots f_r^{a_{ir}},\ i = 1\ldots r$$
for some $\lambda_i \in \BK^*$ and $(a_{ij}) \in \mathrm{GL}_r (\mathbb{Z})$. If we consider an automorphism $\tau: T_r \to T_r$ which is given by the formulas 
$$\tau(t_i) = \lambda_it_1^{a_{i1}}\ldots t_r^{a_{ir}}$$
then $\rho_g = \tau \circ \rho$. 

\begin{definition} We will call the embedding $\rho$ described above \emph{canonical}.

\end{definition}

Note that if $\rho: X \hookrightarrow T_r$ is a canonical embedding, then $\mathbb{K}[X]^*  \simeq \mathbb{K}[T_r]^*$ and $E(X) \simeq E(T_r)$. 
We denote by $\Aut_{X} (T_r)$ the subgroup of $\Aut(T_r)$ which consists of automorphisms of $T_r$ that preserve $X$. There is a natural homomorphism $\Aut_X(T_r) \to \Aut (X)$ which sends an automorphism $\varphi \in \Aut_X(T_r)$ to its restriction $\varphi|_{X}$.

\begin{proposition}\label{Prop1}

Let $X$ be a toral variety and $\rho: X \hookrightarrow T_r$ be a canonical embedding. Then
\begin{enumerate}
\item the homomorphism 
$$\Aut_X(T_r) \to \Aut (X), \ \varphi \to \varphi|_X$$
is an isomorphism; 
\item the subgroup $H(X)$ is the image of the subgroup $\Aut_X(T_r) \cap T_r$ with respect to this isomorphism.
\end{enumerate}
\end{proposition}
\begin{proof}

We denote by $t_1, \ldots, t_r$ coordinate functions on $T_r$ and by $f_1, \ldots, f_r$ the respective invertible regular functions on $X$. Then $[f_1], \ldots, [f_r]$ is a basis of $E(X)$. 

Firstly, we will prove that the homomorphism 

$$\Aut_X(T_r) \to \Aut (X), \ \varphi \to \varphi|_X$$
is surjective.  Let $\overline{\varphi}$ be an automorphism of $X$. Then $\overline{\varphi}$ defines an automorphism of the lattice $E(X)$. Therefore,
$$\overline{\varphi}(f_i) = \lambda_if_1^{a_{i1}}\ldots f_r^{a_{ir}}, \ i = 1,\ldots, r,$$
where $\lambda_i \in \BK^*$ and $(a_{ij}) \in \mathrm{GL}_r(\mathbb{Z})$. We define an automorphism $\varphi$ of $T_r$ by the formulas
$$\varphi(t_i) = \lambda_it_1^{a_{i1}}\ldots t_r^{a_{i_r}}, \ i = 1,\ldots r.$$
Then $\varphi$ preserves $X$ and $\varphi|_X = \overline{\varphi}$.

Now suppose that the image of an automorphism $\psi \in \mathrm{Aut}_X (T_r)$ is a trivial automorphism of $X$. Then $\psi|_X$ defines a trivial automorphism of the lattice $E(X)$. Hence, $\psi$ defines a trivial automorphism of the lattice $E(T_r)$. So $\psi$ has the form
$$\psi(t_i) = \beta_it_i$$
for some $\beta \in \BK^*$. It means that $\psi \in T_r$.  But $T_r$ acts on itself freely. Since $\psi$ preserve all points of $X$ then $\psi$ is a trivial automorphism of $T_r$. So the map
$$\Aut_X(T_r) \to \Aut (X)$$ 
is injective and therefore it is an isomorphism.

It remains to prove the last property. If $\delta \in \Aut_X(T_r) \cap T_r$ then $\delta|_X$ defines a trivial automorphism of $E(X)$. Hence $\delta|_X \in H(X)$. 

Conversely, suppose that $\delta|_X \in H(X)$. Then $\delta$ is given by the formulas
$$\delta(t_i) = \gamma_it_i, \ i = 1,\ldots, r,$$
for some $\gamma_i \in \BK^*$. Therefore $\delta \in \Aut_X(T_r)\cap T_r$. 
\end{proof}

\begin{corollary}\label{CorAut}

Let $X$ be a toral variety and $r = \mathrm{rank}\ E(X)$. Then the group $\Aut(X)$ is isomorphic to a subgroup in $T_r \rtimes \mathrm{GL_r}(\mathbb{Z})$.
\end{corollary}

\begin{remark}
It follows from Proposition \ref{Prop1} that a toral variety $X$ can be embedded in a torus $T_r$ in such a way that any automorphism $X$ can be uniquely extended to an automorphism of $T_r.$ If $X$ is a subvariety of $Z$ it is always natural to ask whether an automorphism of $X$ can be extended to an automorphism of $Z$. Some results concerning this problem can be found in~\cite{Kalim1, Kalim2}.
\end{remark}

\begin{example}\label{ex0}

Let $X$ be a toral variety and $\mathrm{rank}\ E(X) = r$. Then there is a canonical embedding $\rho: X \hookrightarrow T_r$ of $X$ into a torus $T_r$ of dimension $r$. But in some cases it is also possible to embed $X$ into a torus of lower dimension.

Consider 
$$Y = \{ (x, y ) \in (\BK^*)^2 |yx(x-1)(x-2)\ldots (x-k) = 1 \}.$$
It is a closed subvariety of a torus $T_2 = (\BK^*)^2$ so $Y$ is a toral variety. We see that $x, (x-1), \ldots, (x-k)$ are invertible functions on $Y$. We will show that $[x], [x-1], \ldots, [x-k]$ are linearly independent in $E(Y)$. It implies that $\mathrm{rk}\ E(Y) \geq k+1$.

Indeed, otherwise there are $b_0, \ldots, b_k \in \BZ$ and $\lambda \in \BK^*$ such that 
\begin{equation}\label{eq0} x^{b_0}(x-1)^{b_1}\ldots(x-k)^{b_k} = \lambda. \end{equation}

But the polynomial $x^{b_0}(x-1)^{b_1}\ldots(x-k)^{b_k} - \lambda$ is not divisible by  $yx(x-1)(x-2)\ldots (x-k) - 1$ in $\BK[x^{\pm 1}, y^{\pm 1}]$. So Equation \ref{eq0} cannot be hold on $Y$.

\end{example}

\begin{example}

It is also not true that every embedding of a toral variety $X$ with $\mathrm{rank}\ E(X) = r$ into a torus $T_r$ is canonical. 

The embedding $X \hookrightarrow T_r$ is canonical if $[t_1|_X], \ldots, [t_r|_X]$ is a basis of $E(X)$. If we choose $Y \subseteq T_2$ as in Example \ref{ex0} above then the embedding $Y \hookrightarrow T_2 \times T_{r-2} = T_r$, where $ z \to (z, p)$ for some fixed point $p \in T_{r-2}$, is not a canonical embedding. Here the restrictions $t_3|_Y, \dots, t_r|_Y$ are constants so $[t_3|_Y] = \ldots = [t_r|_Y]$ is a neutral element in $E(Y)$.
 
\end{example}

Now let $X$ be a  closed irreducible subvariety in $T_r$ and the embedding $X \hookrightarrow T_r$ be canonical. By Proposition \ref{Prop1} we can identify the group $H(X)$ with the subgroup in $T_r$ which preserves $X$. We will describe the subgroup $H(X)$ as a subgroup in $T_r$.  Let $M \simeq \ZZ^r$ be the lattice of characters of $T_r$. For $m = (m_1, \ldots, m_r) \in M$ by $\chi^m$ we mean the character $t \to t_1^{m_1}\ldots t_r^{m_r}.$ Then each function in $\BK[t_1^{\pm 1}, \ldots, t_r^{\pm 1}]$ is a linear combination of characters. For a function $f = \sum_i \alpha_{m_i} \chi^{m_i} \in \BK[t_1^{\pm 1}, \ldots, t_r^{\pm 1}]$ by \emph{support of $f$} we mean the subset 

$$\supp f = \{m_i \in M |\ \alpha_{m_i} \neq 0 \} \subseteq M.$$

Let $I(X)$ be the ideal of functions in $\BK[t_1^{\pm 1}, \ldots, t_r^{\pm 1}]$ which are equal to zero on $X$. We say that $f \in I(X)$ is \emph{minimal} if there is no non-zero $g \in I(X)$ such that $\supp g \subsetneq \supp f.$ 
\begin{lemma}
Minimal polynomials generate $I(X)$ as a vector space.  

\end{lemma}

\begin{proof}
If $f\in I(X)$ is not minimal then there is a $g\in I(X)$ with $\supp g \subsetneq \supp f$. One can choose a constant $\alpha$ such that $\supp (f - \alpha g) \subsetneq \supp f.$ Applying induction by cardinality of $\supp f$ we see that $g$ and $f - \alpha g$ can be represented as a sums of minimal polynomials. Then $f$ is also a sum of minimal polynomials.
\end{proof}

\begin{definition}
    
\label{def1}

We denote by $M(X)$ the subgroup of $M$ which is generated by Minkowski sums $\supp f + (-\supp f)$ for all minimal $f \in I(X)$.

\end{definition}

\begin{proposition}\label{Prop3}

The subgroup $H(X) \subseteq T_r$  is given by equations $\chi^m(t) = 1$ for all $m\in M(X).$

\end{proposition}

\begin{proof}

Let $h\in H(X)$ and $f = \sum_i \alpha_{m_i}\chi^{m_i}$ be a minimal polynomial in $I(X).$ Then $h\circ f = \sum_i \alpha_{m_i} \chi^{m_i}(h)\chi^{m_i}.$ The ideal $I(X)$ is invariant under the action of $H(X)$. So $h\circ f \in I(X)$. Suppose that there are $a,b \in M(X)$ such that $\alpha_a, \alpha_b \neq 0$ and $\chi^{m_a}(h) \neq \chi^{m_b}(h).$ Then $g = \chi^{m_a}(h)f - h\circ f$ is a non-zero function in $I(X)$ and $\supp g \subsetneq \supp f.$ But $f$ is minimal. So $\chi^{m_a} (h) = \chi^{m_b} (h)$. Therefore, $\chi^{m_a - m_b} (h) = 1$ and this implies that $\chi^{m}(h) = 1$ for all $m \in M(X).$ 

Now consider an element $ t\in T_r$ such that $\chi^m(t) = 1,\ \forall \ m\in M(X).$ Then every minimal polynomial in $I(X)$ is a semi-invariant with respect to $t$. But $I(X)$ is a linear span of minimal polynomials. So $I(X)$ is invariant under the action of $t$. Therefore, $t \in  H(X).$

\end{proof}

At the end of this section, we note that toral varieties over uncountable fields satisfy the following conjecture formulated by Alexander Perepechko and Mikhail Zaidenberg.

\begin{conjecture}[Conjecture 1.0.1 in \cite{PZ}]
If $Y$ is a rigid affine algebraic variety over $\BK$, then the connected component
$\Aut^0(Y)$ is an algebraic torus of the rank not greater than $\dim Y$. 
\end{conjecture}

\begin{corollary}
Suppose that the field $\BK$ is uncountable. Let $X$ be a toral variety over $\BK$. Then $\Aut(X)$ is a discrete extension of an algebraic torus.
\end{corollary}

\begin{proof}
Indeed, if $X$ is a toral variety then the group $\Aut(X)/H(X)$ is isomorphic to  a  subgroup in $\Aut(E(X)) \simeq \mathrm{GL}_r (\BZ)$, where $r$ is the rank of $E(X).$ If $\BK$ is uncountable then $\Aut(X)/H(X)$ is a discrete group. So $\Aut^0(X)$ is contained in $H(X).$ But $H(X)$ is a quasitorus. Therefore, $\Aut^0(X)$  is a torus and the quotient group $\Aut(X)/\Aut^0(X)$ is a discrete group.

\end{proof}

From this point onwards, we always assume that the field $\BK$ is uncountable.

\section{The structure of the automorphism group}

It follows from Corollary \ref{CorAut} that toral varieties are rigid. By \cite[Theorem 2.1]{AGRV}, there is a unique maximal torus in the automorphism group of an irreducible rigid variety.

\begin{theorem}\label{theorem1}
Let $X$ be a toral variety over $\BK$ and $T$ be the maximal torus in $\Aut (X).$ Then $X \simeq Y \times T$ where $Y$ is a toral variety with a discrete automorphism group. 
\end{theorem}
\begin{proof}

Let  $r$ by the rank of the group $E(X)$ and $\rho: X \hookrightarrow T_r$ be a canonical embedding. We denote by $M$ the lattice of characters of $T_r$ and by $M(X)$ the sublattice in $M$ which corresponds to $X$. One can choose a basis $e_1, \ldots, e_r \in M$ such that $b_1e_1, \ldots b_le_l$ is a basis of $M(X)$ for some  $b_1, \ldots, b_l \in \BN$ and $l \leq r$. Denote by ${t_1}, \ldots {t_r}$ coordinates on $T_r$ corresponding to $e_1, \ldots, e_r$. 

Then the equations $\chi^m(t) = 1$ for all $m \in M(X)$ define the subgroup $H(X)$ in $T_r$ which consists of elements of the form

$$(\epsilon_1, \ldots, \epsilon_l, t_{l+1},\ldots, t_r),$$
where $\epsilon_1,\ldots, \epsilon_l$ are the roots of unity of degrees $b_1, \ldots, b_l$ respectively and $t_{l+1}, \ldots t_r \in \BK^*.$ Then the maximal torus in $H(X)$ is the torus 
$$T_{r-l} = \{(1, \ldots, 1, t_{l+1}, \ldots, t_r) \in T_r| t_i \in \BK^*\}.$$
The group $\Aut(X)/H(X)$ is a discrete group. So the maximal torus of $\Aut(X)$ coincides with the maximal torus of the quasitorus $H(X)$ which is $T_{r-l}$.  

All minimal polynomials in $I(X)$ are 
semi-invariant with respect to $H(X)$. This means that minimal polynomials in $I(X)$ are homogeneous with respect to each variable ${t}_{l+1},\ldots, {t}_r.$ Since functions ${t}_i$ are invertible one can choose a set of minimal generators of $I(X)$ which do not depend on ${t}_{l+1}, \ldots, {t}_r.$ It implies that $X \simeq Y \times T_{r-l}$ where $Y$ is a subvariety of $T_l = \{(t_1, \ldots, t_l, 1, \ldots, 1) \in T_r| t_i \in \BK^* \}.$

The variety $Y$ is also a toral variety which is given by the ideal $I(X) \cap \BK[t_1^{\pm 1}, \ldots t_l^{\pm 1}].$ Since the unique maximal torus in $\Aut(X)$ is $T_{r-l}$ the maximal torus in $\Aut(Y)$ is trivial.  

\end{proof}

Let $X$ be a toral variety and suppose that $X\simeq T_s \times Y$ where $Y$ is a toral variety with a discrete automorphism group and $T_s$ is the torus $(\BK^*)^s.$ One can see that $\Aut(X)$ contains the following subgroups. 

There is a subgroup which is isomorphic to $\Aut(Y)$. This subgroup acts naturally on $Y$ and acts trivially on $T_s.$ The subgroup $\mathrm{GL_s}(\BZ)$ acts naturally on $T_s$ and trivially on $Y$. Moreover, there is a subgroup which is isomorphic to $(\BK[Y]^*)^s \simeq (E(Y) \times \BK^*)^s.$ This subgroup acts in the following way. If $f_1,\ldots, f_s \in \BK[Y]^*$ then we can define an automorphism of $T_s \times Y$ as follows
$$(t_1,\ldots, t_s, y) \to (f_1(y)t_1,\ldots, f_s(y)t_s, y).$$

The following theorem was proposed to the authors by Sergey Gaifullin.  

\begin{theorem}\label{Gaifullin}
Let $X \simeq T_s \times Y$ be a toral variety, where $Y$ is a toral variety with a discrete automorphism group. Then
$$\Aut (X) \simeq \Aut (Y) \ltimes (\mathrm{GL}_s (\BZ) \ltimes (E(Y) \times \BK^*)^s).$$
\end{theorem}

\begin{proof}
There is a natural action of $T_s$ on $X$. We see that $\BK[Y]$ is the algebra of invariants of this action. Since $T_s$ is a unique maximal torus in $\Aut(X)$ each automorphism of $T_s \times Y$ preserves $\BK[Y].$ So we obtain a homomorphism 
$$\Phi: \Aut(X) \to \Aut(Y).$$ Let $B$ be the kernel of $\Phi$. The group $\Aut(Y)$ is naturally embedded into $\Aut (T_s \times Y)$ and intersects trivially with $B$. At the same time, $\Aut (Y)$ maps isomorphically to the image of $\Phi$. It implies that
$$\Aut(T_s \times Y) = \Aut(Y) \ltimes B.$$

We denote by $t_1,\ldots, t_s$ coordinate functions on $T_s.$ Then $$\BK[T_s \times Y] \simeq \BK[T_s]\otimes \BK[Y] = \BK[Y][t_1^{\pm 1}, \ldots t_s^{\pm 1}].$$ Let $\phi \in B.$  The algebra $\BK[Y]$ is invariant with respect to $\phi^*$. So for all $t\in T_s$ and $y \in Y$ we have

$$\phi ((t, y)) = (t', y),$$
for some $t'\in T_s.$ Therefore, for each $y\in Y$ the automorphism $\phi$ defines an automorphism $\phi_y : T_s \to T_s.$ Hence, for each $y\in Y$ we have 
$$\phi^*(t_i)(t, y) = t_i(\phi(t, y)) = t_i((\phi_y(t), y)) = f_i(y)t_1^{a_{i1}(y)} \ldots t_s^{a_{is}(y)}$$
for some non-zero constant $f_i(y)$ and a matrix $A(y) = (a_{ij}(y))\in \mathrm{GL}_s (\BZ).$ In reasons of continuity the matrix $A(y)$ is the same for all $y \in Y$ and $f_i : Y \to \BK$ are regular functions on $Y$. Since $f_i(y) \neq 0$ for all $y\in Y$ the functions $f_i$ are invertible. So we have
$$\phi^* (t_i) = f_it_1^{a_{i1}} \ldots t_s^{a_{is}}$$
for some $f_i \in \BK[Y]^*$ and $A\in \mathrm{GL}_s (\BZ).$

Then we have a homomorphism $\overline{\Phi} : B \to \mathrm{GL}_s (\BZ),\ \phi \to A.$ Again, the group $\mathrm{GL}_s(\BZ)$ is naturally embedded into $B$ in the following way. The matrix $(d_{ij}) \in \mathrm{GL}_s(\BZ)$ corresponds to an automorphism
$$(t_1, \ldots, t_s, y) \to (t_1^{d_{11}}\ldots t_{s}^{d_{1s}}, \ldots, t_1^{d_{s1}}\ldots t_{s}^{d_{ss}}, y).$$
The group $\mathrm{GL}_s(\BZ)$ maps isomorphically to $\mathrm{GL}_s (\BZ)$ under~$\overline{\Phi}. $   So
$$B = \mathrm{GL}_s(\BZ) \ltimes \mathrm{Ker}\ \overline{\Phi} .$$
The kernel of $\overline{\Phi}$ consists of automorphisms $\varphi \in \Aut (T_s \times Y)$ which have a form 

$$\varphi(t_1,\ldots, t_s, y) = (f_1(y)t_1,\ldots, f_s(y)t_s, y).$$
for some $f_1, \ldots, f_s \in \BK[Y]^*.$  We see that for all $f_1, \ldots, f_s \in \BK[Y]^*$ this formula defines an automorphism of $T_s \times Y,$ so $\mathrm{\Ker}\ \overline{\Phi} \simeq (\BK[Y]^*)^s \simeq (E(Y)\times \BK^*)^s.$
\end{proof}

\section{The case $\mathrm{rk}\ E(X) = \dim X$}


Let $X$ be a toral variety. Then $\mathrm{rk}\  E(X) \geq \dim X.$ Indeed, suppose that $f_1, \ldots, f_r$ are invertible functions and $[f_1], \ldots, [f_r]$ is a basis in $E(X).$ Then $f_1, \ldots, f_r$ generate $\BK[X].$ So $r \geq \mathrm{tr.deg}\ \BK[X] = \dim X$. 

The following result shows that if $\mathrm{rk}\ E(X) = \dim X$ then $X$ is a torus. Moreover, this is the only case when $\Aut(X)$ acts with an open orbit on $X$.  

\begin{proposition} \label{theorem2}
Let $X$ be a toral variety. Then the following conditions are equivalent.
\begin{enumerate}
    \item $X$ is a torus;
    \item $\mathrm{rk}\ E(X) = \dim  X;$
    \item $\Aut(X)$ acts on $X$ with an open orbit.
\end{enumerate}
\end{proposition}
\begin{proof}
The implication $1) \Rightarrow 2)$ is trivial.

Suppose that $\mathrm{rk}\ E(X) = \dim X.$ Then one can choose invertible functions $f_1, \ldots, f_n$ such that $[f_1],\ldots,[f_n]$ is a basis of $E(X)$. Then $\BK[X]$ is generated by 
$$f_1,f_1^{-1},\ldots,f_n,f_n^{-1}.$$ 
But $f_1,\ldots,f_n$ are algebraically independent, otherwise $\dim X<\mathrm{rk}\ E(X)$. So $\BK[X]$ is isomorphic to the algebra of Laurent polynomials. So we obtain implication $2)~\Rightarrow~1).$

The implication $1) \Rightarrow 3)$ is trivial. Suppose $X$ is a toral variety and $\Aut(X)$ acts on $X$ with an open orbit $U$. 

Let $T$ be the maximal torus in $\Aut(X).$ Since the quotient group $\Aut(X)/T$ is a discrete group the set $U$ is a countable union of orbits of $T.$ Since $\BK$ is uncountable it implies that one of the orbits of $T$ is open in $X.$ Then $\dim X = \dim T.$ By Theorem \ref{theorem1} we have $X \simeq T \times Y$ for some toral variety $Y$. But since $\dim T = \dim T \times Y$ we obtain that $Y$ is a point and $X \simeq T.$

\end{proof}

\section{The case $\mathrm{rk}\ E(X) = \dim X + 1$}

By Theorem \ref{theorem1} any toral variety over algebraically closed uncountable field of characteristic zero is a direct product $T \times Y$ where $T$ is a torus and $Y$ is a toral variety with a discrete automorphism group. By Theorem \ref{Gaifullin} one can find $\Aut(X)$ knowing $\Aut(Y).$ In this section we provide a way to find $\Aut(Y)$ when $\mathrm{rk}\ E(Y) = \dim\ Y + 1.$

Let $Y$ be a toral variety with a trivial maximal torus in $\Aut(Y)$. Let $r$ be the rank of $E(Y)$. We suppose that $r = \dim\ Y + 1.$ 

There is a canonical embedding of  $Y$ into the torus $T_{r}$ as a hypersurface. The variety $T_r$ is factorial so there is an irreducible polynomial $h \in \BK [t_1^{\pm 1}, \ldots, t_r^{\pm 1}]$ such that $I(Y) = (h).$ The polynomial $h$ has a form
$$h = \sum_{m\ \in\ \mathrm{Supp}\ h} \alpha_m \chi^m.$$

Let $M$ be the lattice of characters of $T_r$ and $M(Y)$ be a sublattice in $M$ which corresponds to $Y$; see Definition \ref{def1}. Since the maximal torus in $\Aut(Y)$ is trivial, the rank of the lattice $M(Y)$ is equal to $r$. It means that the elements $m_a - m_b$ with $m_a,m_b \in \mathrm{Supp}\ h$ generate a sublattice of full rank in $M$.

We denote by $\mathrm{GAff}(M, h)$ the group of all invertible integer affine transformations $\varphi$ of $M$ which preserve $\mathrm{Supp}\ h$ and for any linear combination 
$$\sum_{m \in \supp\ h}{a_m}m = 0,$$
where $a_m \in \mathbb{Z}$ and $\sum_m a_m = 0$, the affine transformation $\varphi$ satisfies

\begin{equation}\label{eq2}\prod_{m \in \supp\ h} \left(\alpha_{m}\right)^{a_m} = \prod_{m \in \supp h} \left(\alpha_{\varphi(m)}\right)^{a_m}.
\end{equation}

\begin{theorem}\label{Gaff}
Let $Y$ be a toral variety with trivial maximal torus in $\Aut(Y).$ Suppose that $\mathrm{rk}\ E(Y) = \dim Y + 1.$ Then

$$\Aut(Y)/H(Y) \simeq \mathrm{GAff}(M, h).$$
\end{theorem}

\begin{proof}

Let $\psi$ be an automorphism of $Y$. By Proposition \ref{Prop1} the automorphism $\psi$ can be extended to an automorphism of $T_r$. We denote by $\psi^*$ the respective automorphism of $\BK[t^{\pm 1}_1,\ldots, t^{\pm 1}_r]$. Then $\psi^*$ has the form
$$\psi^*(t_i) = \lambda_i t_1^{a_{i1}}\ldots t_r^{a_{ir}},$$
where $\lambda_i \in \BK^*$ and $(a_{ij}) \in \mathrm{GL}_r (\BZ).$ We denote by $\lambda$ the element $(\lambda_1, \ldots, \lambda_r) \in T_r$ and by $\overline{\psi}$ the automorphism of $M$ which corresponds to the matrix $(a_{ij}).$ Then 
$$\psi^* (\chi^m) = \chi^m(\lambda)\chi^{\overline{\psi}(m)}$$
for all $m \in M.$

The polynomial $\psi^* (h)$ also generates $I(Y).$ So it differs from $h$ by an invertible element of $\BK[t^{\pm 1}_1,\ldots, t^{\pm 1}_r].$ Then 
$$\psi^*(h) = \alpha\chi^vh$$
for some  
$\alpha \in \BK^*$ and $v \in M.$ Therefore, we have an equation

\begin{equation}\label{eq1}\psi^*(h) = \sum_{m \in \mathrm{Supp}\ h} \alpha_m\chi^m(\lambda)\chi^{\overline{\psi}(m)} = \alpha\sum_{m \in \mathrm{Supp}\ h} \alpha_m \chi^{m + v}.\end{equation}

It implies that $\overline{\psi}(m) - v$ belongs to $\mathrm{Supp}\ h$ for all $m \in \mathrm{Supp}\ h$. We define the map $\varphi: M \to M$ by the following formula:
$$\varphi(m) = \overline{\psi} (m) - v.$$
Then $\varphi$ is an affine transformation of $M$ which preserves $\mathrm{Supp}\ h.$ 

We will prove that $\varphi$ belongs to $\mathrm{GAff}(M,h)$. So we consider a linear combination as in (\ref{eq2}): 
$$\sum_{m \in \supp\ h}a_mm = 0,$$
where $a_m \in \mathbb{Z}$ and $\sum_m a_m = 0.$

Equation (\ref{eq1}) can be written as

$$\sum_{m \in \mathrm{Supp}\ h} \alpha_m\chi^m(\lambda)\chi^{\varphi(m)} = \alpha\sum_{m \in \mathrm{Supp}\ h} \alpha_m \chi^{m} = \alpha\sum_{m \in \mathrm{Supp}\ h} \alpha_{\varphi(m)} \chi^{\varphi(m)}.$$

and it implies

$$\frac{\alpha_{m_{1}}\chi^{m_{1}}(\lambda)}{\alpha_{m_{2}}\chi^{m_{2}}(\lambda)} = \frac{\alpha_{m_{1}}}{\alpha_{m_{2}}} \chi^{m_{1} - m_{2}}(\lambda) = \frac{\alpha_{\varphi(m_{1})}}{\alpha_{\varphi(m_{2})}}$$
for all $m_1, m_2 \in \supp\ h.$

We fix some $m_0 \in \supp\ h$. Then we have

$$\prod_{m \in \supp\ h} \left(\alpha_{\varphi(m)} \right)^{a_m} = \prod_{m \in \supp\ h} \left(\frac{\alpha_{\varphi(m)}}{\alpha_{\varphi(m_0)}} \right)^{a_m} = \prod_{m \in \supp\ h} \left(\frac{\alpha_{m}}{\alpha_{m_0}} \chi^{m-m_0}(\lambda) \right)^{a_m} = $$

$$ = \prod_{m \in \supp\ h} \left(\frac{\alpha_{m}}{\alpha_{m_0}} \right)^{a_m}( \chi^{\sum_ma_m(m-m_0)}(\lambda)) = \prod_{m \in \supp\ h} \left(\frac{\alpha_{m}}{\alpha_{m_0}} \right)^{a_m} = \prod_{m \in \supp\ h} \left(\alpha_{m} \right)^{a_m}$$

So $\varphi \in \mathrm{GAff} (M, h).$ Then we obtain a homomorphism 
$$\eta: \Aut(Y) \to \mathrm{GAff} (M, h),\ \psi \to \varphi.$$
Moreover we see that the kernel of $\eta$ is $H(Y).$ Now we will show that $\eta$ is surjective. 

Let $\varphi \in \mathrm{GAff} (M, h)$ and $f_1, \ldots, f_r$ be a basis in $M(Y).$ Again we fix some $m_0 \in \supp h$. Then there are  $a_{m,j} \in \BZ$ for $m \in \supp h$ such that
$$f_j = \sum_{m \in \supp\ h} a_{m,j}(m - m_{0}).$$
There is a $\lambda \in T_r$ such that
$$\chi^{f_j} (\lambda) = \prod_{m \in \supp\ h}\left(\frac{\alpha_{m}}{\alpha_{m_{0}}}\right)^{-a_{m,j}}\prod_{m \in \supp\ h}\left(\frac{\alpha_{\varphi(m)}}{\alpha_{\varphi(m_{0})}}\right)^{a_{m,j}}$$
for all $j = 1,\ldots, r.$ 

Let $d\varphi$ be the linear part of $\varphi,$ i.e. $d\varphi(m) = \varphi(m) - \varphi(0).$ We define an automorphism $\psi^*$ of $\BK[t_1^{\pm 1}\ldots t_r^{\pm 1}]$ by the following rule:
$$\psi^*(\chi^m) = \chi^m(\lambda)\chi^{d\varphi(m)}.$$
Let us check that $\psi^*$ preserves $I(Y).$ We have
$$\psi^*(h) = \sum_{m \in \supp\ h} \alpha_m\chi^m(\lambda)\chi^{d\varphi(m)}.$$
We denote $\varphi(0)$ by $v$. Then 
$$\varphi(m) = d\varphi(m) + v$$
and 
$$\chi^v\psi^*(h) = \sum_{m \in \supp\ h} \alpha_m\chi^m(\lambda)\chi^{\varphi(m)}.$$
We see that $\mathrm{Supp}\ \chi^v\psi^*(h) = \mathrm{Supp}\ h.$ We will show that there is $\alpha \in \BK$ such that
$$\chi^v\psi^*(h)  = \alpha h.$$

For any $b, c \in \mathrm{Supp}\ h$ there are numbers $d_j \in \BZ$ such that
$$b-c = \sum_{j = 1}^r d_jf_j.$$
So 
$$ \frac{\alpha_b\chi^b(\lambda)}{\alpha_c\chi^c(\lambda)} = \frac{\alpha_c}{\alpha_b}\chi^{b-c}(\lambda)=\frac{\alpha_b}{\alpha_c}\chi^{\sum_j d_jf_j}(\lambda)=$$
\begin{equation}\label{eqa}= \frac{\alpha_b}{\alpha_c}(\prod_{j = 1}^r\chi^{f_j}(\lambda))^{d_j}=\frac{\alpha_b}{\alpha_c}\prod_{m ,j = 1}\left( \frac{\alpha_{m}}{\alpha_{m_{0}}}\right)^{-d_ja_{ m,j}}\prod_{m,j}\left( \frac{\alpha_{\varphi(m)}}{\alpha_{\varphi(m_{0})}}\right)^{d_ja_{m,j}}.\end{equation}

We have a combination 
$$0 = b - c - \sum_j d_jf_j = b - c - \sum_{m, j} d_ja_{m,j}(m - m_{0}) = b - c - \sum_{m,j}d_ja_{m,j }m + (\sum_{m, j}d_ja_{m, j})m_0.$$
The sum of all coefficients in the last sum is equal to $0$. Since $\varphi \in \mathrm{Gaff}(M, h)$ we obtain

\begin{equation}\label{eqb}\frac{\alpha_b}{\alpha_c}\prod_{m,j}\left(\frac{\alpha_{m}}{\alpha_{m_{0}}}\right)^{-d_ja_{m,j}}= \frac{\alpha_{\varphi(b)}}{\alpha_{\varphi(c)}}\prod_{m ,j}\left(\frac{\alpha_{\varphi(m)}}{\alpha_{\varphi(m_{0})}}\right)^{-d_ja_{m,j}}.\end{equation}
It follows from equations \ref{eqa} and \ref{eqb} that
$$\frac{\alpha_b\chi^{b}(\lambda)}{\alpha_c\chi^c(\lambda)} = \frac{\alpha_{\varphi(b)}}{\alpha_{\varphi(c)}}.$$

So the coefficients of the polynomials $\chi^v\psi^*(h)$ and $h$ are proportional. Then there is an $\alpha \in \BK$ such that $\chi^v\psi^*(h) =\alpha h.$ Hence $\psi^*(h) = \alpha \chi^{-v}h \in I(Y)$. Therefore, $\psi^*$ preserves $I(Y)$ and defines an automorphism $\psi$. It is a direct check that $\eta(\psi) = \varphi$. So $\eta$ is surjective. 

\end{proof}

\begin{corollary}
Let $Y$ be a toral variety with trivial maximal torus in $\Aut(Y).$ Suppose that $\mathrm{\rk}\ E(Y) = \dim Y + 1.$ Then $\Aut(Y)$ is a finite group.
\end{corollary}

\begin{proof}
Indeed, the group $H(Y)$ is finite in this case. As we mentioned before the sublattice $M(Y)$ is of full rank and generated by the finite set $\supp\ h + (- \supp\ h).$ Then any affine transformation of $M$ is uniquely defined by the image of the set $\supp\ h +( - \supp\ h).$ Therefore, the group $\mathrm{GAff}(M, h)$ is finite. Then the group $\Aut(Y)$ is also finite.

\end{proof}

It is natural to formulate the following question.

\begin{conjecture}
    
Let $Y$ be a toral variety with trivial maximal torus in $\Aut(Y).$ Is $\Aut(Y)$ a finite group?
\end{conjecture}
Note that this is not true for a general rigid variety. One can find a counterexample in~\cite{Perep1}.

At the end, we give three examples illustrating Theorem \ref{Gaff}.

\begin{example}
Let $Y$ be the affine line $\mathbb{A}^1$ without two points. Then $Y$ is isomorphic to an open set of the torus $\BK^*$:
$$Y = \{ t \in \BK^* | t \neq 1\} \subseteq \BK^*.$$
Hence, $Y$ can be given in $(\BK^*)^2$ as the set of solutions of the equation
$$h = t_1(t_2-1) - 1 = 0,\ (t_1,t_2) \in (\BK^*)^2.$$
So $Y$ is a toral variety. We have 
$$\BK[Y] = \BK[t_1^{\pm 1}, t_2^{\pm 1}]/(t_1(t_2 - 1) - 1) \simeq \BK[t_2^{\pm 1}]_{t_2 - 1},$$
where $\BK[t_2^{\pm 1}]_{t_2 - 1}$ denotes the localization of $\BK[t_2^{\pm 1}]$ at $t_2 - 1$. Hence, all invertible elements of $\BK[Y]$ have the form $\lambda (t_2-1)^at_2^b = \lambda t_1^at_2^b$ where $\lambda \in \BK^*.$ Therefore, $[t_1],[t_2]$ is a basis of $E(Y)$. So the rank of  $E(Y)$ is equal to $\mathrm{dim}\ Y + 1$ and the embedding $ Y \hookrightarrow (\BK^*)^2$ as a set of zeros 
$$h = t_1(t_2-1) - 1 = t_1t_2 - t_1 - 1 =  0,\ (t_1,t_2) \in (\BK^*)^2$$
is a canonical embedding. We can apply Theorem~\ref{Gaff} to find $\Aut(Y)$.

Let $M \simeq \BZ^2$ be the lattice of characters of $(\BK^*)^2$. The set $\mathrm{Supp}\ h$ consists of points $m_0 = (0, 0),\ m_1 = (1, 0), m_2 = (1, 1);$ see Figure \ref{fig1}.

\begin{figure}
\centering
\includegraphics[width=0.5\linewidth]{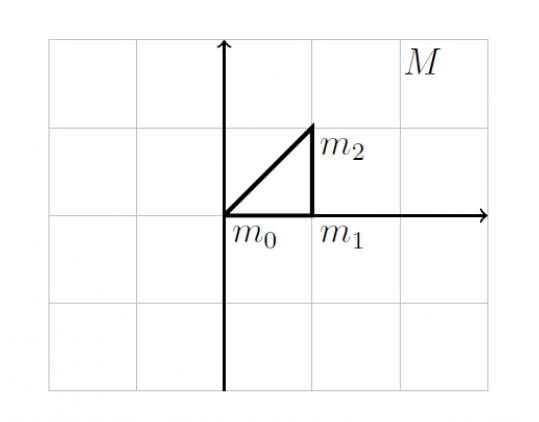}
\caption{\label{fig1}$\mathrm{Supp}\ (t_1t_2 - t_1 - 1)$}

\end{figure}

We see that the lattice $M(Y)$ contains elements $(1, 0), (0, 1),$ so $M(Y) = M.$  Therefore, $H(Y)$ is a trivial group.

A linear combination
$$a_0m_0 + a_1m_1 + a_2m_2 = (a_1 + a_2, a_2)$$
with $a_0 + a_1 + a_2 = 0$ is equal to zero if and only if $a_0 = a_1 = a_2 = 0.$ But then equations (\ref{eq2}) are trivial. By affine transformations of $M$ we can permute all points in $\mathrm{Supp}\ h.$ Therefore, 
$$\Aut(Y) \simeq \mathrm{GAff}(M, h) \simeq S_3.$$

The answer looks natural since the affine line without two points is the projective line without three points. 

In this case $\Aut(Y)$ is generated by the automorphisms $\psi_1, \psi_2$, where
$$\psi_1((t_1, t_2)) = (-t_1t_2, t_2^{-1}), \ \psi((t_1, t_2)) = (-t_2, t_1^{-1}t_2^{-1}).$$
\end{example}

\begin{example}

Now let $Y$ be the set of solutions of the equation 
$$Y = \{ (t_1, t_2, t_3) \in (\BK^*)^3 | \ h = t_3(t_1^2 + t_2^2 - 1) - 1 = 0 \} \subseteq (\BK^*)^3.$$
Then $Y$ is a toral variety and 
$$\BK[Y] = \BK[t_1^{\pm 1}, t_2^{\pm 1}, t_3^{\pm 1}/(t_3(t_1^2 + t_2^2 - 1) - 1) = \BK[t_1^{\pm_1}, t_2^{\pm 1}]_{t_1^2 + t_2^2 - 1}.$$
Therefore, $[t_1], [t_2], [t_3]$ is a basis of $E(Y)$ and the embedding of $Y$ in $(\BK^*)^3$ is a canonical embedding. 

We have $h = t_3(t_1^2 + t_2^2 - 1) - 1 = t_1^2t_3 + t_2^2t_3 - t_3 - 1$ and 
$$\supp h = \{ m_0 = (0, 0, 0), m_1 = (0, 0, 1), m_2 = (2, 0, 1), m_3 = (0, 2, 1) \} \subseteq M \simeq \BZ^3. $$
The vectors $(2, 0, 0), (0, 2, 0)$ and $(0, 0, 1)$ form a basis of $M(Y)$. Then the group $H(Y) \subseteq (\BK^*)^3$ consists of elements 
$$H(Y) = \{(\pm 1, \pm 1, 1) \in (\BK^*)^3 \} \simeq \BZ/2\BZ \times \BZ/2\BZ .$$

The group of invertible affine transformations of $M$ which preserve $\supp h$ is isomorphic to $S_3$ and permutes points $m_1, m_2, m_3$ preserving $m_0$. The sum

$$a_0m_0 + a_1m_1 + a_2m_2 + a_3m_3 = (2a_2, 2a_3, a_1 + a_2 + a_3)$$
with $a_0 + a_1 + a_2 + a_3 = 0$ is equal to zero if and only if $a_0 = a_1 = a_2 = a_3 = 0$. So the equations (\ref{eq2}) are trivial and $\mathrm{GAff}(M, h) \simeq \Aut(Y)/H(Y) \simeq S_3.$ 

The group $\Aut(Y)$ is generated by $H(Y)$ and the automorphisms $\psi_1$ and $\psi_2$ which are defined by the formulas:

$$\psi_1 ((t_1, t_2, t_3)) = (t_2, t_1, t_3), \ \psi_2((t_1, t_2, t_3)) = (-t_2^{-1}, it_1t_2^{-1}, -t_2^2t_3).$$

One can check that $\psi_1$ and $\psi_2$ generate the subgroup in $\Aut(Y)$ which is isomorphic to $S_3$ and trivially intersects with $H(Y)$. So

$$\Aut(Y) \simeq H(Y) \rtimes S_3.$$

The automorphism $\psi_2$ do not commute with the element $(1, -1, 1) \in H(Y)$. Therefore, $\Aut(Y)$ is not a direct product of $H(Y)$ and $S_3$. 
\end{example}

\begin{remark}
It is natural to ask, is it true that, under the conditions of Theorem \ref{Gaff}, we have $\Aut(Y) \simeq H(Y) \rtimes \mathrm{Gaff}(M, h)?$ The authors do not know the answer to this question.
\end{remark}
\section*{Acknowledgement}

The authors are grateful to Segrey Gaifullin for useful discussions. Also, we would like to thank Ivan Arzhantsev for helpful remarks and comments.

\end{document}